\documentclass[11pt]{amsart}

\usepackage{graphicx, amsmath, amssymb, amscd, amsthm, euscript, psfrag, amsfonts}

\newtheorem{theorem}{Theorem}[section]
\newtheorem{lemma}[theorem]{Lemma}
\newtheorem{proposition}[theorem]{Proposition}
\newtheorem{corollary}[theorem]{Corollary}

\theoremstyle{definition}

\theoremstyle{remark}
\newtheorem{remark}[theorem]{Remark}

\numberwithin{equation}{section}

\newcommand{\e}{\epsilon}

\newcommand{\cO}{\mathcal O}

\newcommand{\field}[1]{\mathbb{#1}}

\newcommand{\R}{\field{R}}
\newcommand{\N}{\field{N}}

\providecommand{\abs}[1]{\lvert#1\rvert}
\providecommand{\Abs}[1]{\Bigl\lvert #1 \Bigr\rvert}
\providecommand{\norm}[1]{\lVert#1\rVert}

\newcommand{\q}{\mathbb{Q}}

\newcommand{\z}{\mathbb{Z}}
\renewcommand{\q}{\mathbb{Q}}

\renewcommand{\c}{\mathbb{C}}
\newcommand{\br}{\mathbb{R}}

\newcommand{\cS}{\mathcal S}
\newcommand{\la}{\langle}
\newcommand{\ra}{\rangle}

\newcommand{\SL}{\operatorname{SL}}

\newcommand{\bp}{\begin{pmatrix}}
\newcommand{\ep}{\end{pmatrix}}

\newcommand{\inv}{^{-1}}

\newcommand{\SO}{\operatorname{SO}}

\newenvironment{smallpmatrix}{\left(\begin{smallmatrix}}{\end{smallmatrix}\right)}

\newcommand{\vol}{\operatorname{vol}}

\newcommand{\PSL}{\op{PSL}}
\newcommand{\bi}{\begin{itemize}}
\newcommand{\bs}{\backslash}

\newcommand{\ei}{\end{itemize}}

\newcommand{\op}{\operatorname}

\renewcommand{\setminus}{\smallsetminus}

\newcommand{\cl}[1]{\overline{#1}}

\newcommand{\ba}{\backslash}
\newcommand{\cB}{\mathcal{B}}

\newcommand{\cK}{\mathcal{K}}

\newcommand{\cW}{\mathcal{W}}
\newcommand{\G}{\Gamma}

\begin{document}

\title[Limits of translates of divergent geodesics]{Limits of translates of divergent geodesics and Integral points on one-sheeted hyperboloids}

\author{Hee Oh}

\address{Mathematics department, Brown university, Providence, RI, U.S.A
and Korea Institute for Advanced Study, Seoul, Korea} 
\email{heeoh@math.brown.edu}
\thanks{Oh was partially supported by NSF
   Grant \#0629322.}

\author{Nimish A. Shah}
\address{Department of Mathematics,
The Ohio State University, Columbus, OH 43210, U.S.A}
\email{shah@math.osu.edu}


\subjclass[2010]{11N45, 37A17, 22E40 (Primary)}
\keywords{Counting lattice points, mixing of geodesic flow, decay of matrix coefficients}

 \thanks{Shah was partially supported by NSF Grant \#1001654.}

\begin{abstract} 
For any non-uniform lattice $\G$ in $\SL_2(\br)$,
 we describe the limit distribution of orthogonal translates of a {\it divergent} geodesic
in $\G\ba \SL_2(\br)$. As an application, for a quadratic form
$Q$ of signature $(2,1)$, a lattice $\G$ in its isometry group, and $v_0\in \br^3$ with $Q(v_0)>0$,
 we compute the asymptotic (with a logarithmic error term) of the number
of points in a discrete orbit $v_0\G$ of norm at most $T$, when the stabilizer of $v_0$ in $\G$
is finite. Our result in particular implies that for any non-zero integer $d$, the smoothed count
for number of
integral binary quadratic forms with discriminant $d^2$ and with
coefficients bounded by $T$ is asymptotic to $c\cdot  T \log T +O(T)$.
\end{abstract}

\maketitle

\section{Introduction}

\subsection{Motivation} Let $Q\in \z[x_1, \cdots, x_n]$ be a homogeneous polynomial and
set $V_m:=\{x\in \br^n: Q(x)=m\}$ for an integer $m$.
It is a fundamental problem to understand the set
$V_m(\z)=\{x\in \z^n: Q(x)=m\}$ of integral solutions.

In particular, we are interested in the asymptotic
of the number $N(T):=\#\{x\in V_m(\z): \|x\|<T\}$ as $T\to \infty$,
where $\|\cdot \|$ is a fixed norm on $\br^n$.

The answer to this question depends quite heavily on the geometry of the ambient space $V_m$.
We suppose that the variety $V_m$ is homogeneous, i.e., there exist a connected semisimple real
algebraic group $G$ defined over $\q$ and a $\q$-rational representation $\iota: G\to \SL_n$ such that
$V_m=v_0. \iota(G)$ for some non-zero $v_0\in \q^n$.

Let $\G<G(\q)$ be an arithmetic subgroup preserving $V_m(\z)$. 
By a theorem of Borel and Harish-Chandra \cite{BH}, the co-volume of $\G$ in $G$ is finite and
there are only finitely many $\G$-orbits in $V_m(\z)$. Hence
understanding the asymptotic of $N(T)$ is reduced to the orbital counting problem on 
$\#(v_0\Gamma \cap B_T)$ for $B_T=\{x\in V_m: \|x\|<T\}$ and $v_0\in V_m(\z)$.

\begin{theorem} \label{drs} Set $H$ to be the stabilizer subgroup of
$v_0$ in $G$. Suppose that $H$ is either a symmetric subgroup or a maximal $\q$-subgroup of $G$.
If the volume
of $(H\cap \G)\ba H$ is finite, i.e., if $H\cap \G$ is a lattice in $H$, we have  
\begin{equation*}
 \#(v_0\Gamma \cap B_T) \sim \frac{\op{vol}_{H}(H\cap \G
\ba H)}{\op{vol}_G(\G\ba G)}\op{vol}_{H\ba G} (B_T) \end{equation*} where the volumes on
$H, G$ and $v_0G\simeq H\ba G$ are computed with respect to
invariant measures chosen compatibly; that is, $d\op{vol}_G=d\op{vol}_H\times d\op{vol}_{H\ba G}$ locally.
\end{theorem}
This theorem was first proved by Duke, Rudnick, Sarnak \cite{DukeRudnickSarnak1993}
 when $H$ is symmetric and Eskin and McMullen gave a simplified proof  in
\cite{EskinMcMullen1993}. When $H$ is a maximal $\q$-subgroup, it is proved by Eskin, Mozes and Shah in \cite{EskinMozesShah1996}.

As apparent from the main term of the asymptotic, 
it is crucial to assume  $\op{vol}(H\cap \G\ba H)<\infty$ in Theorem \ref{drs}. 
The main aim of this paper is to break this barrier; to investigate the counting problem 
in the case when $\op{vol}(H\cap \G\ba H)=\infty$. 

We focus on the case when $Q$
is a quadratic form of signature $(n-1,1)$ with $n\ge 3$ and $G$
is the special orthogonal group of $Q$. In this situation,
the case of $\op{vol}(H\cap \G \ba H)=\infty$ for $H=\op{Stab}_G(v_0)$ arises only when $n=3$ and $Q(v_0)=m>0$, that is,
when the variety $V_m=\{x\in \br^3: Q(x)=m\}$ is a one-sheeted hyperboloid.
To prove this claim, note first that 
if $H$ is a non-compact simple Lie group, 
 then any closed $\G\ba \G H$ in $\G\ba G$ must be of finite volume 
 by Dani~\cite{Dani1979} and Margulis \cite{MargulisICM1991} (see also \cite{Shah1991}).
 Any non-compact stabilizer $H$ of $v_0\in \br^n$ in $G$ is either locally isomorphic to
$\SO(n-2,1)$ (which is a simple Lie group
except for $n=3$) or a compact extension of a horospherical subgroup.
Since any orbit of a horospherical subgroup is either compact or dense in $\G\ba G$
(cf.~\cite{Dani1986}), it follows that
the case of
$\op{vol}(H\cap \G \ba H)=\infty $
 arises only when $H\simeq \SO(1,1)$; hence
  $n=3$ and $Q(v_0)>0$.

In the next subsection, we state our main theorem in a greater generality, not necessarily
in the arithmetic situation.
\subsection{Counting integral points on a one-sheeted hyperboloid}  \label{sec:1.2}
Let $Q(x_1, x_2, x_3)$ be an real quadratic form of signature $(2,1)$.
Denote by $G$ the identity component of the special orthogonal group $\SO_Q(\br)$.
Let $\G<G$ be a lattice and $v_0\in \br^3$ be such that
 $Q(v_0)>0$ and the orbit $v_0\G$ is discrete.
As before, we fix a norm $\|\cdot \|$ on $\br^3$ and set 
$B_T:=\{x\in v_0G:\|w\|<T\}$.

To present our theorem with a best possible error term, we consider the following
smoothed counting function: fixing a non-negative function $\psi\in C_c^\infty(G)$ with integral one,
let $$\tilde N_T:=\sum_{v\in v_0\G} (\chi_{B_T}* \psi )(v)$$ 
where $\chi_{B_T}*\psi (x)=\int_{G}\chi_{B_T}(x g)\psi(g) \;dg$, $x\in v_0G$, is the convolution of the characteristic function of $B_T$ and $\psi$.
Note that
$\tilde N_T \asymp \# (v_0\G\cap B_T)$ in the sense that their ratio is in between two
 uniform constants for all $T> 1$.

 Denoting by $H\simeq \SO(1,1)^\circ$ the one-dimensional stabilizer subgroup of $v_0$ in $G$,
note that $\op{vol}(H\cap \G\ba H)<\infty$ if and only if $H\cap \G$ is infinite.
In order to state our theorem,
we write $H$ as a one-parameter subgroup $\{h(s):s\in \br\}$ so that the Lebesgue measure
 $ds$ defines a Haar measure on $H$:
$\int_{-\log T}^{\log T} ds=\vol_H(\{h(s):|s|<\log T\})$.

\begin{theorem}\label{main} If the volume
of $(H\cap \G)\ba H$ is infinite,  we have the following:
\begin{enumerate}
 \item As $T\to \infty$,
$$N_T\sim
\frac{\int_{-\log T}^{\log T} ds}{\op{vol}_G(\G\ba G)}\op{vol}_{H\ba G}(B_T) $$
where $d\op{vol}_G=ds\times d\op{vol}_{H\ba G}$ locally.
\item for $T \gg 1$, 
$$ \tilde N_T =c\cdot T\log T +O(T)$$
where $c=\lim_{T\to \infty}\frac{2\op{vol}_{H\ba G}(B_T)}{T\op{vol}_G(\G\ba G)}$. 
\end{enumerate}
 \end{theorem}

We note that when $\op{vol}(H\cap \G\ba H)<\infty $, $\tilde N_T =c\cdot T + O(T^{\alpha})$
for $0<\alpha <1$  is obtained in  \cite{DukeRudnickSarnak1993}.
We believe, as suggested by Z. Rudnick to us, that
$\tilde N_T = c \cdot T\log T +c' \cdot T +O(T^\alpha) $ for some $c'>0$ and $0<\alpha<1$ and hence
the order of the second term for $\tilde N_T$ cannot be improved.

Theorem \ref{main} can be generalized to the orbital counting for more general representations of $\SL_2(\br)$ (see section
\ref{sec:SL2-gen-rep}).

\begin{remark} In the case when $Q=x_1^2+x_2^2-d^2x_3^2$ for $d\in \z$, $v_0=(1,0,0)$,
and $\G=\SO_Q(\z)$, it was pointed out in \cite{DukeRudnickSarnak1993} 
that an elementary number theoretic computation of \cite{Scourfield1961}
 leads to the asymptotic 
 $$\#\{(x_1, x_2, x_3)\in v_0\Gamma: \sqrt{x_1^2+x_2^2+d^2x_3^2}< T\} =
     c \cdot T  \log T +O(T\log(\log T)) .$$
However this deduction seems to work only for this very special case; for instance, we are not aware of any
other approach than ours which can deal with non-arithmetic situtations.
\end{remark}

\subsection{Arithmetic case and Integral binary quadratic forms}
In the arithmetic case, Theorem \ref{main} together with Theorem \ref{drs} implies the following:
\begin{corollary}\label{arith}
 Let $Q(x_1, x_2, x_3)$ be an integral quadratic form with signature $(2,1)$. Suppose
that for some $v_0\in \z^3$ with $Q(v_0)>0$, the stabilizer subgroup of $v_0$ is isotropic over $\q$.

Then there exists $c=c(\|\cdot \|) >0$ such that as $T\to \infty$,
$$\# \{x\in \z^3: Q(x)=Q(v_0), \|x\|<T\}\sim c\cdot  T \log T .$$
\end{corollary}

For a binary quadratic form $q(x,y)=ax^2+bxy+cy^2$, its discriminant $\op{disc}(q)$ is defined to
be $b^2-4ac$. The group $\SL_2(\br)$ acts on the space of binary
quadratic forms by $(g. q)(x, y)= q(g^{-1}(x,y))$ and preserves the discriminant.
For $d\in \z$, denote by $\cB_d(\z)$
  the space of integral binary quadratic forms 
with discriminant $d$.
Note that $\cB_d(\z)\ne \emptyset$ if and only if $d$ congruent to $0$ or $1 \mod 4$.
Now $d$ is a square {\em if and only if\/} the stabilizer of every $q\in \cB_d(\z)$
in $\SL_2(\z)$ is infinite {\em if and only if\/} every $q\in \cB_d(\z)$ is decomposable over $\z$.
 (cf. \cite{BV}).

Therefore Corollary \ref{arith}
 implies the following:
\begin{theorem}.
For any non-zero square $d\in \z$, there exists $c_0>0$ such that
$$\#\{q\in \cB_d(\z): \op{disc}(q)=d, \|q\|<T\} \sim  c_0 \cdot T\log T  $$
where $\|ax^2+bxy+cy^2\|=\|(a,b,c)\|$.
 \end{theorem}

\subsection{Orthogonal translates of a divergent geodesic}

Let $G=\SL_2(\R)$ and $\Gamma$ be a non-uniform lattice in $G$. For
$s\in\R$, define
\begin{equation} \label{ha}
h(s)=\Bigl[\begin{smallmatrix} \cosh(s/2) & \sinh(s/2) \\
  \sinh(s/2) & \cosh(s/2)\end{smallmatrix} \Bigr], \ 
a(s)=\Bigl[\begin{smallmatrix} e^{s/2} & 0 \\
 0 & e^{-s/2}\end{smallmatrix}\Bigr]
\end{equation}
and set $H=\{h(s):s\in \br\}$.

In the case when the orbit $\G\ba\G H$ is closed and
of finite length, the limiting distribution of the translates 
$\G\ba\G H a(T)$ as $T\to\infty$ is described by the unique $G$-invariant probability measure $d\mu(g)=dg$ on $\G\ba G$ \cite{DukeRudnickSarnak1993}, that is, if $s_0$ is the period
of $\G\cap H\ba H$, then for any $\psi \in C_c(\G\ba G)$,
$$\lim_{T\to \pm \infty} \frac{1}{s_0} \int_{s=0}^{s_0} \psi (h(s)a(T))ds =\int_{\G\ba G}\psi \; dg.$$

Similarly,
  understanding the limit of the translates $\G\ba \G H a(T)$ when
$\G\ba \G H$ is divergent (and of infinite length) is the main new ingredient  in
our proofs of Theorem \ref{main}.

\begin{theorem}  \label{thm:main} Let $x_0\in \Gamma\ba G$ and suppose
  that $x_0h(s)$ diverges as $s\to+\infty$, that is,
  $x_0h(s)$ leaves every compact subset for all sufficiently large $s\gg 1$.
For a given compact subset $\cK\subset \G\ba G$,
there exist $c=c(\cK)>0$ and $M=M(\cK)>0$ such that for any $\psi \in C^\infty(\G\ba G)$
with support in $\cK$, we have, as $|T|\to \infty$,
$$ \int_{0}^{\infty}\psi (x_0h(s)a(T))ds= \int_{0}^{T+M}\psi (x_0h(s)a(T))ds = |T|\int \psi\; d\mu +O ( 1) $$
where the implied constant depends only on $\cK$ and a Sobolev norm
of $\psi$. 
\end{theorem}

\begin{remark}\rm Consider the hyperbolic plane $\mathbb H^2$.
A parabolic fixed point for $\G$ 
is a point in the geometric boundary $\partial_\infty(\mathbb H^2)$  fixed by a parabolic element of $\G$.
If $\mathcal F\subset \mathbb H^2$ is a finite sided Dirichlet region for $\G$, then
the parabolic fixed points of $\G$ are precisely the $\G$-orbits of 
vertices of $\overline{\mathcal F}$ lying
in $\partial_\infty(\mathbb H^2)$.
Let $\pi: G\to \mathbb H^2$ denote the orbit map $g\mapsto g(i)$.
For $x_0=\G g_0\in \G\ba G$,
the image $\pi(g_0 H)$ is a geodesic in $\mathbb H^2$ with two endpoints
$g_0H(+\infty):=\lim_{s\to \infty} \pi(g_0h(s))$ and $g_0H(-\infty):=
\lim_{s\to -\infty} \pi(g_0h(s))$ in $\partial_\infty(\mathbb H^2)$.
We remark that $x_0h(s)$ diverges as $s\to + \infty$ (resp. $s\to -\infty$)
if and only if
$g_0H(+\infty)$ (resp. $g_0H(-\infty)$)  is a parabolic fixed point for $\G$ (cf. Theorem \ref{thm:cusps}).
\end{remark}

 \begin{corollary}\label{corodiv}  Suppose that $x_0 H$ is closed and non-compact. For any $\psi\in
  C_c(\Gamma\ba G)$,
  \[
  \lim_{T\to\pm\infty} \frac{1}{2\abs{T}} \int_{
  -\infty}^\infty
  \psi(x_0h(s)a(T))\,ds=\int_{\Gamma\ba G} \psi\,d\mu.
  \]\end{corollary}

\subsection*{Acknowledgements} We thank Zeev Rudnick for insightful comments.

\section{Structure of cusps in $\Gamma\ba G$ and divergent trajectory}\label{sec2}
Let $G=\SL_2(\R)$ and $\Gamma$ be a non-uniform lattice in $G$.
We will keep the notation for $h(s)$ and $a(s)$ from \eqref{ha} in the introduction.
 Let
$$N=\{\begin{smallpmatrix}1 & s \\ 0 &
  1\end{smallpmatrix}:s\in\R \}\quad\text{ and } \quad U=wNw^{-1}$$
  where
  $w=\bigl[\begin{smallmatrix} \cos(\pi/4) & \sin(\pi/4) \\
  -\sin(\pi/4) & \cos(\pi/4) \end{smallmatrix}\bigr]$.
  Note that $h(s)=w a(s)w^{-1}$ for all $s\in \br$.
   For
$\eta>0$, let $$H_\eta=\{h(s):s/2 >-\log \eta\} .$$

 Let $K=\SO(2)=\{g\in G: g g^t=I\}$. Then
the multiplication map $U\times H\times K\to G$: $(u,h,k)\mapsto uhk$ is a diffeomorphism.

The following classical result may be found at \cite[Thm. 0.6]{GarlandRaghunathan1970} or \cite{DaniSmillie1984}:
\begin{theorem}  \label{thm:cusps} There exists a finite set
  $\Sigma\subset G$ such that the following holds:
  \begin{enumerate}
  \item $\Gamma\ba\Gamma \sigma U$ is compact for every $\sigma\in\Sigma$.
  \item For any $\eta>0$, the set
    \begin{equation*}  
\cK_\eta:=\Gamma\ba G \setminus
      \bigcup_{\sigma\in\Sigma} \G\ba \G \sigma UH_\eta K
    \end{equation*}
    is compact; and any compact subset of $\Gamma\ba G$ is contained in $\cK_\eta$
    for some $\eta>0$.

  \item There exists $\eta_0>0$ such that for $i=1,2$, if  $\sigma_i\in \Sigma$, $u_i\in U$, $h_i\in H_{\eta_0}$, and   
   $\Gamma \sigma_1 u_1 h_1 k_1 = \Gamma \sigma_2 u_2 h_2 k_2 $, 
    then $\sigma_1=\sigma_2$, $k_1=\pm k_2$ and $h_1=h_2$.
  \end{enumerate}
\end{theorem}

Consider the standard representation of $G=\SL_2(\R)$ on $\R^2$: $((v_1, v_2), g)\mapsto (v_1,v_2)g$. Let $\norm{\cdot}$ denote the Euclidean norm on $\R^2$. 
Let
\[
p=(0,1) w^{-1} =(-\sin(\pi/4),
  \cos(\pi/4)).\]
Then $pU=p$, and $ph(s)=(0,1) a(s) w^{-1} = e^{-s/2}p$ for all $s\in \br$.  Also 
\begin{equation}  \label{eq:eta}
g\in UH_\eta K \Leftrightarrow \norm{pg}<\eta.
\end{equation}

\begin{proposition}[{Dani~\cite{Dani:Crelle}}]  \label{prop:divergent}
Let $x_0\in \Gamma\ba G$ be such
  that the trajectory $\{x_0 h(s):s\geq 0\}$ is divergent. Then there
  exist $\sigma_0\in {\pm I}\Sigma$, $s_0\in \R$ and $u\in U$ such that
  $x_0 =\G \sigma_0 uh(s_0)$.
\end{proposition}

\begin{proof}
  By Theorem~\ref{thm:cusps}, there exists $s_1>0$ and $\sigma\in\Sigma$ such that 
  $x_0h(s) =\Gamma \sigma UH_{\eta_0/2}K$ for all $s\geq s_1$. Let $g_1\in UH_{\eta_0/2}K$ be such that $x_0h(s_1)=\Gamma\sigma g_1$. We claim that $pg_1\in \R p$. If not, then $\norm{ph(s)}\to \infty$ as $s\to\infty$, and hence there exists $s>0$ such that $\eta_0/2 \norm{pg_1h(s)} < \eta_0$. By \eqref{eq:eta}, 
$g_1h(s)\in uhk$ for some $u\in U$, $h\in H_{\eta_0}$ and $k\in K$. Therefore 
\[
\Gamma \sigma uhk=\Gamma \sigma g_1h(s)=x_0h(s_1+s)\in \Gamma \sigma UH_{\eta_0/2}K.
\]
 By Theorem~\ref{thm:cusps}(3), we have that $h\in H_{\eta_0/2}$. But then $\norm{pg_1h(s)}=\norm{puhk}<\eta_0/2$, a contradiction. Therefore our claim that $pg_1\in \R p$ is valid. Hence $g_1= u_1h(s)\{\pm I\}$ for some $u_1\in U$ and $s/2\geq -\log(\eta_0/2)$. Thus $x_0h(s_1)=\Gamma \sigma u_1 h(s) \{\pm I\}$, and hence $x_0=\Gamma\sigma_0 u_1 h(s-s_1)$, where $\sigma_0={\pm I}\sigma$. 
\end{proof}  

\begin{proposition} \label{prop:zero}  Let  $x_0\in \Gamma\ba G$ be such
  that the trajectory $\{x_0 h(s):s\geq 0\}$ is divergent. Let $\cK\subset \G\ba G$ be a compact subset.
There exists $M_1=M_1(\cK)> 0$ such that
 $$x_0 h(s) a(T)\not\in \cK $$  for any $T\in\R$ and $s>0$ satisfying $s>|T|+M_1$. In particular,
  for any $f\in C(\G\ba G)$ with support inside $\cK$,
   \[
  \int_{0}^\infty f(x_0h(s) a(T))\,ds=\int_0^{\abs{T}+M_1}
  f(x_0 h(s)a(T) )\,ds.
  \]
\end{proposition}

\begin{proof}
 By Proposition~\ref{prop:divergent}, $x_0=\Gamma \sigma_0 uh(s_0)$ for some
  $\sigma_0\in {\pm \Sigma}, u\in U, s_0\in \br$. By Theorem~\ref{thm:cusps}(2),  let $\eta>0$ be such that $\cK \subset \cK_\eta$. Let
$M_1=-s_0-2\log(\eta)$. Since $s-|T|>-s_0-2\log \eta$, we have
  \begin{equation} \label{eq:sT}
\begin{array}{ll}
    \norm{ p uh(s_0)h(s)a(T)}
    &=\norm{  ph(s+s_0)a(T) }\\
    &=e^{-(s+s_0)/2}\norm{pa(T)} \\
    & < e^{-(s+s_0)/2}e^{{\abs{T}/2}}\\
    &= e^{-(s+s_0-\abs{T})/2}<\eta.
\end{array}
  \end{equation}
  Therefore by \eqref{eq:eta}, $uh(s_0)h(s)a(T)\in  UH_{\eta}K$, and hence 
\[
x_0h(s)a(T)\in \Gamma \sigma_0UH_\eta K\subset \Gamma\bs G\setminus K_\eta.
\] 
\end{proof}

\section{Uniform mixing on compact sets} 

Let $G=\SL_2(\br)$ and $\G<G$ be a lattice.
Let $\mu$ denote the $G$-invariant probability measure on $\G\ba G$.
For an orthonormal basis $X_1, X_2, X_3$ of $\mathfrak{sl}(2, \R)$
with respect to an $Ad$-invariant scalar product, and $\psi \in C^\infty(\G\ba G)$,
we consider the Sobolev norm
$$\cS_m(\psi )=\max\{\|X_{i_1}\cdots X_{i_j}(\psi )\|_2:1\le i_j \le 3 ,0\le j\le m\} .$$

The well-known spectral gap property for $L^2(\Gamma\ba G)$ says that the trivial
representation is isolated (see \cite[Lemma 3]{Bekka1998}) in the Fell topology of the unitary
dual of $G$.
It follows that there exist $\theta>0$ and $c>0$ such that
for any $\psi_1, \psi_2\in C^\infty(\Gamma\ba G)$ 
with $\int\psi_i d\mu=0$, $\cS_1(\psi_i)<\infty$ and for any $T>0$,
\begin{equation}\label{theta}
|\langle a(T)\psi_1, \psi_2\rangle | \le  c e^{-\theta |T|}\cS_1(\psi_1)\cS_1(\psi_2)\end{equation}
(cf. \cite{CowlingHaggerupHowe}, \cite{Venkatesh2007})

Write $\cO_\e=\{g\in G: \|g- I\|_\infty\le \e\}$.
For a compact subset $\cK\subset\Gamma\ba G$, let $0<
\e_0(\cK)\le 1 $ be the injectivity radius of $\cK$, that is,
$\e_0(\cK)$ is the supremum of $0< \e\le 1$ such that the multiplication map
$ \cK \times  \cO_{\e} \to \Gamma\ba G$
is injective.

For $s\in \br$, let
$$ n_+(s)=\begin{smallpmatrix} 1 & 0 \\ s & 1 \end{smallpmatrix}
        \quad \text{ and } \quad n_-(s)=\begin{smallpmatrix}1 & s \\ 0 &
        1\end{smallpmatrix}. $$

  \begin{theorem} \label{thm:mixing}
  Let $\cK\subset\G\ba G$ be a compact subset and $\eta>0$.
   There exists $c=c(\cK)>0$ such that for any $\psi \in C^\infty(\Gamma\ba G)$ with support in $\cK$,
     for any $|T|\geq 1$, $x\in \cK$, and $0<r_0<\e_0(\cK)$, we have
    \begin{equation}
      \Abs{\int_0^{r_0} \psi (xn_\nu(s)a(T) )\,ds- r_0  \int \psi \,d\mu}
      \leq  c (\cS_3(\psi) +1) e^{-\theta_0 |T|}
    \end{equation} for some $\theta_0$ depending only on the spectral gap for $L^2(\G\ba G)$.
    Here and in what follows, the sign $\nu=+$ if $T>0$ and
    $\nu=-$ if $T<0$.
  \end{theorem}
\begin{proof} Consider the case when $T>0$ and hence $\nu=+$. The other case can be proved similarly.

Let $\e>0$.
Fix a non-negative function $\rho_\e\in C^\infty_c(N^+)$ which is $1$ on
$n_+[0,r_0]$ and $0$ outside $n_+[-\e, r_0+  \e]$.
Let $N^{\pm}=\{n_{\pm}(s):s\in \br\}$
and $ W_\e:=AN^-\cap \cO_\e$.
Let $\mu_0$ denote the right invariant measure on $AN^-$ such that
$d\mu_0 \otimes dn=d\mu$.
We choose a non-negative
function $\phi_\e \in C^\infty(AN^-)$ supported inside $W_\e$ and $\int \phi_\e d\mu_0=1$.

If we a consider a function $\tau_{x,\e}$ on $\G\ba G$ which is defined to be
$\tau_{x,\e}(y):= \rho_\e(n_+(s) )\phi_\e(w) \in C^\infty(\Gamma\ba G)$
if $y=xn_+(s)w\in x \;\text{supp}(\rho_\e) W_\e$ and $0$ otherwise,
then
$\cS_1(\tau_{x,\e})\ll  \e^{-3}$ where the implied constant is independent of $x$ and
\begin{multline}\label{idd}\langle a(T)\psi, \tau_{x,\e}\rangle
=\int_{y\in \Gamma\ba G } \psi(ya(T)) \tau_{x,\e}(y) d\mu(y)
\\= \int_{u\in W_\e, s\in \R} \psi( xn_+(s)wa(T)) \phi_\e(w) \rho_\e (n_+(s))d\mu_0(w)ds .\end{multline}

As $\cK$ is compact, the $C^1$-norm of $f$ supported inside $\cK$ is bounded above by
a uniform multiple of $\cS_3(\psi )$ (cf. \cite[Thm 2.20]{Aubinbook}) and hence for some $c_1>0$,
\begin{equation}\label{sup}
\max\{\|\psi\|_\infty, C_\psi\}<c_1 \cS_3(\psi)\end{equation}
where $C_\psi$ is the Lipschitz constant of $\psi$.

Since for all $T>0$,
$W_\e a(T)\subset a(T)\cO_{2\e} , $
we have for all $w\in W_\e$ and $T\gg 1$,
$$\Abs{\psi(x n_+(s)wa(T)) -\psi(x n_+(s)a(T))}\le 2 c_1 \cS_3(\psi) \e .$$

Hence by \eqref{idd},
\begin{align*} &\Abs{\langle a(T)\psi, \tau_{x,\e}\rangle -\int_{s\in \R} \psi(x n_+(s) a(T)) \rho_\e(n_+(s)) ds }\\
 &=\Bigl| \int_{w,s} \psi(x n_+(s) w a(T)) \phi_\e(w) \rho_\e (n_+(s))d\mu_0(w)ds\\
& \qquad\qquad\qquad \qquad -\int_{s\in \R} \psi(x n_+(s) a(T)) \rho_\e(n_+(s)) ds \Bigr| \\
&\ll 2 c_1 \cS_3(\psi) \e\|\rho_\e\|_1\le 2 c_1 \cS_3(\psi) \e (r_0+2\e).
\end{align*}

Since
$$ \Abs{\langle a(T)\psi, \tau_{x,\e}\rangle-\int \psi d\mu\cdot
 \|\rho_\e\|_1 } \ll e^{-\theta T} \e^{-3}\cS_1(\psi), $$
we deduce
\begin{align*} & \Abs{\int_{0}^{r_0} \psi(xn_+(s) a(T)) ds -r_0\int \psi d\mu}
\\ &\le \Abs{\int_{s\in \br} \psi(xn_+(s) a(T))\rho_\e(n_+(s)) ds -r_0\int \psi d\mu \|\rho_\e\|_1 } +4c_1  \e  \cS_3(\psi)\\
&\le \Abs{\langle a(T)\psi, \tau_{x,\e}\rangle
- r_0\int \psi d\mu\cdot  \|\rho_\e\|_1} + 6 c_1 \e \cS_3(\psi) \\
&\le 6c_1 \e \cS_3(\psi)  +c\cdot e^{-\theta T} \e^{-3} \cS_3(\psi)\end{align*}
for some $c>0$.
Hence for $\e=e^{-\theta T/4}$ and some $c_2>0$,
$$\Abs{\int_{0}^{r_0} \psi(xn_+(s)a(T)) ds - r_0 \int \psi d\mu }\le c_2 (\cS_3(\psi) +1)  e^{-\theta T/4} . $$
 \end{proof}

\section{Translates of divergent orbits}
Let $x_0\in \G\ba G$ be such that $x_0h(s)$ diverge as $s\to \infty$. 

  \begin{theorem} \label{thm:combine2} 
  For any $|T|>1$ and any $\psi\in C_c^\infty(\G\ba G)$
\[
     \int_0^{|T|} \psi(x_0h(s) a(T) )\,ds =| T| \int
      \psi\,d\mu +O(1)\cS_3(\psi).
     \]
  \end{theorem}

\begin{proof}
Let $R_0=-\log \eta_0$. Due to Proposition~\ref{prop:divergent}, replacing $x_0$ by another point in $x_0H$, we may assume that $x_0=\Gamma \sigma_0 h(R_0)$. For any $S>0$, $\norm{ph(R_0)h(S)a(S)}\in [\eta_0/\sqrt{2}, \eta_0]$. Hence $x_0h(R_0)h(S)a(S)\in K_{\eta_0/\sqrt{2}}$. 

 Let $r_0$ be the injectivity radius of $K_{\eta_0/\sqrt{2}}$, that is,
$r_0=\e_0( K_{\eta_0/\sqrt{2}})$.
 Let $S_0=0$, and choose $S_i$ such that
 $ r_0e^{-S_i}\le \delta_i:=S_{i+1}-S_i\le 2r_0e^{-S_i}$
 for each $i$. We will choose $S_i=\log(2r_0 i +1)$ for each $i$. Then $x_0h(S_i)a(S_i)\in K_{\eta_0/\sqrt{2}}$. We put $R_i=T-S_i$.

We will express $x_0h([S_i,S_{i+1}])a(T)=x_ih^{a(S_i)}([0,\delta_i])a(R_i)$, where $x_i=x_0h(S_i)a(S_i)$ and $h^{a(S_i)}(s)=a(-S_i)h(s)a(S_i)=n(e^{S_i}s/2)w_i(s)$, and $\abs{w_i(s)}=O(e^{-2S_i})$. Note that $r_0/2\leq e^{S_i}\delta_i/2 \leq r_0$. 

By Theorem \ref{thm:mixing}, we have
\[
\int_0^{r_0} \psi(x_i n(s) a(R_i)) ds - r_0\int\psi d\mu = \cS_3(\psi)\cdot O(e^{-\theta_0R_i})
\]
and hence
\[
\int_{S_i}^{S_{i+1}} \psi(x_0h(s)a(T)) ds = \frac{\delta_i}{r_0}\int_0^{r_0}\psi(x_in(s)a(R_i))ds +
 \cS_3(\psi)\cdot O(e^{-2S_i}\delta_i).
\]

Let $k=k(T)$ be such that $S_k\le T<S_k +r_0e^{-S_k}$. Therefore, since $\delta_i r_0^{-1}  \leq 2e^{-S_i}$,

\begin{align*}
&\int_0^T \psi(xh(s)a(T)) ds =\sum_{i=0}^{k-1}\int_{S_i}^{S_{i+1}} \psi(xh(s)a(T)) ds +O(e^{-S_{k}})
\\ & = \sum_{i=0}^{k-1} \delta_i \frac{1}{r_0} \int_0^{r_0}\psi(x_in(s)a(T))ds +
\cS_3(\psi)\cdot O(e^{-2S_i}\delta_i)+O(1)
\\&= \sum_{i=0}^{k-1} \delta_i\mu(\psi) + \sum_{i=0}^{k-1} \delta_i r_0^{-1}\cS_3(\psi)\cdot O(e^{-\theta_0R_i})
+\cS_3(\psi)\cdot O(e^{-2S_i}\delta_i) +O(1)
\\
&= T \mu (\psi) +  O( \sum_{i=1}^{k-1} e^{-S_i}e^{-\theta_0 R_i} +\sum_{i=1}^k e^{-3S_i} )\cS_3(\psi) +O(1)
\\&= T \mu (\psi) +  O( e^{-\theta_0 T}\sum_{i=0}^{k-1} e^{(1-\theta_0) S_i} +\sum_{i=0}^{k-1} e^{-3S_i} )\cS_3(\psi) +O(1).
\end{align*}

Since $S_i=\log(2r_0 i +1)$, 
$0<T-S_k<2e^{-T}$ implies that $k<\frac{e^T-1}{2r_0}<k+1$, and hence
$$\sum_{i=0}^{k-1} e^{-3S_i}\ll \sum_{i=1}^{k-1} \frac{1}{(2r_0 i+1)^3}=O(k^{-2}+1)=O(e^{-2T}+1)<\infty$$
and
$$\sum_{i=0}^{k-1} e^{(1-\theta_0) S_i } \ll \int_{0}^{e^T} \frac{1}{(2r_0 x+1)^{1-\theta_0}} dx=O(e^{\theta_0 T}).$$
 Hence $$ e^{-\theta_0 T}\sum_{i=0}^{k-1} e^{(1-\theta_0) S_i} +\sum_{i=0}^{k-1} e^{-3S_i} =O(1).$$
Therefore $$\int_0^T \psi(xh(s)a(T)) ds = T \mu (\psi) +  O(1)\cS_3(\psi) .$$
\end{proof}

Theorem~\ref{thm:main} follows from the following:

\begin{theorem}\label{mct2} Let $x_0h(s)$ diverge as $s\to \infty$.
For a given compact subset $\cK\subset \G\ba G$, and $\psi\in C^\infty(\G\ba G)$
with support in $\cK$, we have
$$  \int_{0}^{\infty}\psi(x_0h(s)a(T))ds=|T|\cdot \int \psi\;d\mu+O(1)\cS_3(\psi). $$
\end{theorem}
\begin{proof} Since $x_0 h(s)$ diverges as $s\to\infty$, by Proposition \ref{prop:zero},
there exists $M_1=M_1(\cK) >0$ such that
\begin{align*}
&  \int_{0}^{\infty}\psi (x_0h(s)a(T))ds=
\int_{0}^{|T|+M_1}\psi (x_0h(s)a(T))ds \\
& = (|T|+M_1) \int \psi \; d\mu + O (1)\cS_3(\psi)\\
&=  |T| \int \psi \; d\mu + O (1)\cS_3(\psi).
\end{align*}
\end{proof}

By a similar argument, we also deduce the following:

\begin{corollary}   \label{cor:negativeside}
If $x_0h(s)$ diverges as $s\to -\infty$, then
$$\int_{-\infty}^{0}\psi(x_0h(s)a(T))ds=|T| \int \psi d\mu+O(1)\cS_3(\psi) .$$
\end{corollary}

\begin{lemma}\label{div}
If $x_0h(\br)$ is closed and non-compact,
then $x_0h(s)$ diverges as $s\to \pm \infty$.
\end{lemma}

\begin{proof}
We use a well-known fact that for a closed subgroup $H$ of a locally
compact second countable group $G$ and a discrete subgroup $\G$ of $G$,
if $\G H$ is closed in $G$, then the canonical projection
map $H\cap \G\ba H\to \G\ba G$ is a proper map (cf. \cite{OhShahGFH}).
Since $x_0h(\br)$ is non-compact and $h(\br)$ is one-dimensional
with no non-trivial finite subgroups, the stabilizer of $x_0$ in $h(\br)$ is trivial.
 Therefore the map $h(\br)\to \G\ba G$ given by $h\to x_0h$
 is a proper injective map. This implies that $x_0h(s)$ diverges as $s\to \pm \infty$.
\end{proof}

\begin{proof}[Proof of Corollary \ref{corodiv}]
 As the set $C_c^\infty(\G\ba G)$ is dense in $C_c(\G\ba G)$, the claim follows from Lemma~\ref{div}, Theorem~\ref{thm:combine2}, and Corollary~\ref{cor:negativeside}. 
\end{proof}

\section{Counting: Proof of Theorem \ref{main}}
Let $Q$ be a real quadratic form in $3$ variables of
signature $(2,1)$ and $\G_0$ a lattice in the identity component
$G_0$ of $\SO_Q(\br)$. We assume that $v_0\G_0$ is discrete for some vector $v_0\in \br^3$
with $Q(v_0)=d>0$ and that the stabilizer $H_0$ of $v_0$ in $G_0$ is finite.

It suffices to prove Theorem \ref{main} in the case
when $Q=x^2+y^2-z^2$ and $v_0=(\sqrt d, 0,0)$ by the virtue of Witt's theorem.

Consider the spin double cover map $\iota:  G:=\SL_2(\br)\to G_0$ given by
 \begin{equation*}\begin{smallpmatrix} a&b\\c&d\end{smallpmatrix}\mapsto
 \begin{smallpmatrix}
\frac{a^2-b^2-c^2+d^2}{2}& {ac-bd} &\frac{a^2-b^2+c^2-d^2}{2}\\
{ab-cd}& {bc+ad}& {ab+cd}\\
\frac{a^2+b^2-c^2-d^2}{2}& {ac+bd} &\frac{a^2+b^2+c^2+d^2}{2}
\end{smallpmatrix}.
\end{equation*}

For $s\in \br$, we set $$h(s)=\begin{smallpmatrix} \cosh(s/2) & \sinh(s/2) \\
  \sinh(s/2) & \cosh(s/2)\end{smallpmatrix};\quad\text{and}  \quad a(s)=\begin{smallpmatrix} e^{s/2} & 0 \\
  0 & e^{-s/2}\end{smallpmatrix} .$$

Recall that $H:=\{h(s):s\in \br\}$, $A:=\{a(t):t\in \br\}$ and $K_1:=\{k(\theta):\theta\in[0,2\pi]\}$, here $K_1$ is half of the circle group.
Observing that
  $$\iota(h(s))=\begin{smallpmatrix}   1&0&0\\
  0&\cosh s &\sinh s\\ 0&\sinh s&\cosh s\end{smallpmatrix} \quad\text{ and}\quad
   \iota(a(t))=\begin{smallpmatrix}   \cosh t& 0 & \sinh t\\
  0& 1& 0\\ \sinh t&0 &\cosh t\end{smallpmatrix}, $$
 the subgroup $\tilde H:=\pm H$ is the stabilizer of
$v_0$ in $G$.
We have a generalized Cartan decomposition $G=\tilde HAK_1$ in the sense that
every $g$ is of the form $hak$ for unique $h\in \tilde H, a\in A, k\in K_1$.
 And for $g=h(s)a(t)k$, $d\mu(g)=\sinh(t)dsdtdk$
defines a Haar measure on $G$, where $dk=(1/2\pi) dk(\theta)$, and $ds$, $ dt$ and $d\theta$ are Lebesgue measures.
As $v_0G=\pm H\ba G \simeq A\times K_1$, $\sinh(t)dtdk$ defines an invariant measure
on $v_0G$. We consider the volume forms on $G$ and $v_0G$
with respect to these measures. Via the map $\iota$, these define invariant measures
on $G_0$ and $v_0G_0$ as well.

Denote by $\G$ the pre-image of $\G_0$ under $\iota$. Then $\op{Stab}_\G(v_0)=
\tilde H\cap \G=\{\pm I\}$.

For each $T>1$,
define a function on $\G\ba G$:
$$F_T(g):=\sum_{\gamma\in {\pm I}\ba \G}\chi_{B_T}(v_0 \gamma g) .$$

\begin{proposition}\label{inner} For any $\Psi\in C_c^\infty(\G\ba G)$,
 $$\la F_{ T}, \Psi\ra=\frac{T\log T\mu(\Psi) } {\op{vol}(\G\ba G)} \cdot
2\int_{K_1} \frac{1}{\|v^+k\|}dk
 +O(T)\cS_3(\psi)$$
 where
$v^{\pm}=\frac{\sqrt d}{2} (e_1\pm e_3) $.
Here the implied constant depends only on $\mathcal S_3(\Psi)$ and the support of $\Psi$.
\end{proposition}

\begin{proof}
Then $v_0= v^++v^-$ and $v_0 a(t)= e^tv^++e^{-t}v^-$.
Since $B_T=\{ v_0a(t)k: \|v_0a(t)k\|<T\,t\in\R,\,k\in K_1\}$, we have
\begin{align*} 
&\la F_{ T}, \Psi\ra =\int_{\G\ba G}\sum_{\gamma\in {\pm I}\ba \G}\chi_{B_T}(v_0\gamma g)
\Psi(g)d\mu(g)\\ & =
  \int_{k\in K_1} \int_{\|v_0a(t) k\|< T }
   \left( \int_{h(s) \in \pm I \ba \tilde H }\Psi(h(s) a(t) k)ds\right) \sinh(t) dt dk \\
   &= \int_{k\in K_1}\int_{\|v_0a(t)k\|< T }
   \left( \int_{s\in \br }\Psi (h(s)a(t) k)ds\right)  \sinh(t) dt dk .\end{align*}

 Since $v_0\G$ is discrete and $H\cap \G$ is trivial, it follows that
  $\G\ba \G H$ is closed and non-compact in $\G\ba G$. Now fix any $k\in K_1$. 
  Hence by Theorem \ref{mct2} and Lemma \ref{div},
 \begin{align*}\label{coo} 
& \int_{t\gg 1, \|v_0a(t)k\|< T}\left( \int_{s\in \br }\Psi(h(s)a(t) k)ds\right)   \sinh(t)  dt\\
  &=  \frac{1}{\op{vol}(\G\ba G)} \int_{t\gg 1, e^t \|v^+ k \|< T +O(1) }
  (2t \mu(\psi) +O(1)\cS_3(\psi)) (e^t/2+O(1))   dt\\
 & = \frac{T\log T \mu(\Psi)}{\op{vol}(\G\ba G)\cdot \|v^+k\|} +O(T)\cS_3(\psi).
 \end{align*}

  Similarly,
   \begin{align*} & \int_{t\ll -1,
    \|v_0a(t)k\|< T}\left( \int_{s\in \br }\Psi(h(s)a(t) k)ds\right)
      \sinh(t) dt\\&=\int_{t\gg 1,
    \|v_0a(-t)k\|< T}\left( \int_{s\in \br }\Psi(h(s)a(-t) k)ds\right)
      \sinh(t) dt
      \\  & =\frac{1}{\op{vol}(\G\ba G)} \int_{t\gg 1, e^t \|v^- k \|< T + O(1) }
  (2t\mu(\psi) +O(1)\cS_3(\psi)) (e^t/2+O(1))   dt\\
 &= \frac{T\log T\mu(\Psi) }{\op{vol}(\G\ba G) \|v^-k\|}  +O(T)\cS_3(\psi). \end{align*}
Since $v^-k(\pi)=-v^+$,
\[
\int_{k\in K_1} \norm{v^-k}\inv dk = \int_{k\in K_1} \norm{v^+k(\pi)k}\inv dk=\int_{K_1}\norm{v^+k}\inv dk.
\]
The required formula can be deduced in a straightforward manner from this. 
\end{proof}

Fix a non-negative function $\psi\in C_c^\infty(G)$ whose support injects to $\G\ba G$  and
with integral $\int \psi(g)\;d\mu(g)=1$.
Consider a function $\xi_T$ on $\br^3$ defined by
$$\xi_T(x)=\int_{g\in G}\chi_{B_T}(xg)\psi(g) d\mu(g) .$$
Then
the sum $\sum_{\gamma\in \pm I\ba \G}\xi_T(v_0\gamma)$ is a smoothed over counting
satisfying 
$$\sum_{\gamma\in \pm I\ba \G}\xi_T(v_0\gamma)\asymp \#v_0\G\cap B_T .$$

\begin{theorem} As $T\to \infty$,
$$\sum_{\gamma\in \pm I\ba \G}\xi_T(v_0 \gamma) =
\frac{2 T\log T } {\op{vol}(\G\ba G)} \cdot \int_{k\in K_1} \frac{1}{\|v^+k\|}dk
+O( T)\cS_3(\psi).$$
 \end{theorem}
\begin{proof} It is not hard to verify that
$$\sum_{\gamma\in \pm I\ba \G}\xi_T(v_0\gamma)=\la F_T, \Psi\ra $$
  where $\Psi(\G g)=\sum_{\gamma\in \G}\psi(\gamma g)$.
Therefore the claim follows from Proposition \ref{inner}.
\end{proof}

\begin{theorem}\label{mctwo} For $T\gg 1$, we have
$$\#\{w\in v_0\G: \|w\|<T\}=
  \frac{2 T\log T }{\op{vol}(\G\ba G)} \int_{K_1}\frac{1}{ \|w^+k\|} dk (1+(\log T)^{-\alpha})$$
where $\alpha=-1/5.5$. 
\end{theorem}

\begin{proof}
Note that $F_T(I)=\#\{w\in v_0\G: \|w\|<T\}$.
For each $\e>0$,
let $\cO_\e=\{g\in G:\|g-I\|_\infty\le \e\}$.
There exists $0< \ell \le 1 $ such that
for all small $\e>0$,\begin{equation}
\label{oe} \cO_{\ell \e} B_T\subset B_{(1+\e)T},\quad B_{(1-\e)T}\subset \cap_{u\in \cO_{\ell \e} }u B_T .\end{equation}
Let $\psi^\e$ be a non-negative smooth function on $G$
 supported in $\cO_{\ell \e}$ and
with integral $\int \psi^\e d\mu=1$ and define
$\Psi^\e\in C_c^\infty(\G\ba G)$ by
$\Psi^\e(\G g):=\sum_{\gamma\in \G}\psi^\e(\gamma g)$.

Using \eqref{oe}, we have
$$\la F_{(1-\e)T}, \Psi^\e\ra \le F_T(I)\le
\la F_{(1+\e)T}, \Psi^\e\ra .$$

Therefore by Proposition~\ref{inner}
\begin{align*} 
\la F_{(1\pm \e) T}, \Psi^\e\ra &= \frac{2 T\log T }{\op{vol}(\G\ba G)} \int_{K_1}\frac{1}{ \|w^+k\|} dk
+O(\e T\log T) + O(\mathcal S_3 (\Psi^\e) T)\\
&= \frac{2 T\log T }{\op{vol}(\G\ba G)} \int_{K_1}\frac{1}{ \|w^+k\|} dk (1+(\log T)^{-1/5.5}, \end{align*}
where the last equality follows because $\mathcal S_3 (\Psi^\e)=O(\e^{-4.5})$, and if we put $\epsilon=(\log T)^{-1/5.5}$ then 
\[
O(\mathcal S_3(\Psi^\e)T)=O(\epsilon T\log T)= (T\log T)(\log T)^{-1/5.5}.
\]
\end{proof}

\begin{proof}[Proof of Theorem~\ref{main}]
The above computation  in the proof of Proposition~\ref{inner} also shows that
\begin{equation} 
\op{vol}(B_T)=\int_{k\in K_1} \int_{\|v_0a(t)k\|<T }
  \sinh(t)  dt dk 
=T \int_{k\in K}\frac{1}{\|v^+k\|} dk +O(\log T).
\end{equation}
From Theorem~\ref{mctwo}, it follows that
\begin{equation}\label{f2}
F_T(I)= \frac{2\log T \op{vol}(B_T)} {\op{vol}(\G\ba G)} (1+O(\log T)^{-\alpha})). \end{equation}
Since $F_T(I)=\#(v_0\G\cap B_T)$, this completes the proof.
\end{proof}

\section{Orbital counting for general representations of $\SL_2(\R)$}  \label{sec:SL2-gen-rep}

Let $G=\SL_2(\R)$ and $\Gamma$ be a non-uniform lattice in $G$. For
$s\in\R$, define
\begin{equation*} 
h(s)=\Bigl[\begin{smallmatrix} \cosh(s/2) & \sinh(s/2) \\
  \sinh(s/2) & \cosh(s/2)\end{smallmatrix} \Bigr], \ 
a(s)=\Bigl[\begin{smallmatrix} e^{s/2} & 0 \\
 0 & e^{-s/2}\end{smallmatrix}\Bigr], \  k(\theta)=\Bigl[\begin{smallmatrix} \cos(\theta/2) & -\sin(\theta/2) \\ \sin(\theta/2) & \cos(\theta/2)\end{smallmatrix}\Bigr]
\end{equation*}
Put $H=\{h(s):s\in\R\}$, $A^+=\{a(t):t>0\}$, and $K_1=\{k(\theta):\theta\in [0,2\pi]\}$, here $K_1$ is half of the circle group. Put $w_0=k(\pi)$. Then $\{\pm I\}\bs G=HA^+K_1\cup Hw_0A^+K_1$, $w_0\inv h(s)w_0 = h(-s)$ and $w_0\inv a(t)w_0= a(-t)$.

Let $V$ be any finite dimensional representation of $G$ and $v_0\in G$ be such that $H$ is the stabilizer subgroup of $v_0$ in $G$, i.e.,
 $H=G_{v_0}$ where $G_{v_0}=\{g\in G: v_0g=v_0\}$.
Assume that $V$ is linearly spanned by $v_0G$. Then if $e^{mt}$ is the highest eigenvalue for $a(t)$-action on $V$, then $m\in \N$, and the $G$ action factors through $\{\pm I\}\bs G=\PSL_2(\R)\cong \SO(2,1)^0$. 

For example, let $V_m$ denote the $(2m+1)$-dimensional space of real homogeneous polynomials of degree $2m$ in two variables, and consider the standard right action of $g\in \SL(2,\R)$ on $P(x,y) \in V_m$ by $(Pg)(x,y)=P((x,y)g)$, where $(x,y)\bigl[\begin{smallmatrix} a & b \\ c & d \end{smallmatrix}\bigr]=(ax+cy,bx+dy)$. Let $v_0(x,y)=(x^2-y^2)^m$. Then $G_{v_0}=H\cW$, where $\cW=\{\pm I\}$ if $m$ is odd and $\cW=\{\pm I, \pm w_0\}$ if $m$ is even. Moreover, $\{P\in V_m:Ph=P,\,\text{ for all }h\in H\}=\R P_0$. A general finite dimensional representation of $G$ with a nonzero $H$-fixed vector is a direct sum of such irreducible representations, and $v_0$ is a sum of one nonzero $H$-fixed vector from each of the irreducible representations; we assume that $V$ is a span of $v_0G$.

\begin{theorem} \label{thm:gen-count}
Let $V$, $v_0$ and $m$ be as above. Suppose that $\Gamma$ is a lattice in $G$, $v_0\Gamma$ is discrete, and $\Gamma_{v_0}:=\Gamma\cap G_{v_0}$ is finite. Let $\norm{\cdot}$ be any norm on $V$, and $v_0^+=\lim_{t\to\infty} v_0a_t/\norm{v_0 a_t}$. Let $C$ be an open subset of $\{v\in V:\norm{v}=1\}$ such that $\Theta=\{\theta\in[0,2\pi]: v_0^+k(\theta)\in \R C\}$
has positive Lebesgue measure, and $\{\theta\in [0,2\pi]: v_0^+ k(\theta)\in \R (\cl{C}\setminus C)\}$ has zero Lebesgue measure. Then for $T\gg 1$,  
\begin{align}
\# (v_0\Gamma &\cap [0,T]C )\\
&=\frac{4 (2\pi)\inv \int_{\Theta}\norm{v_0^+ k(\theta)}^{-1/m}\,d\theta} {\abs{\Gamma_{v_0}}\cdot \vol_G(\Gamma\bs G)}\times \frac{\log T}{m}T^{1/m} (1+(\log T)^{-\alpha})
\nonumber
\end{align}
where $\vol_G$ is given by the Haar integral $dg=\sinh(t)dtdsd\theta$ on $G$, where $g=h(s)a(t)k(\theta)$, and $\alpha=\frac{1}{5.5}$.                    

Moreover, if $C\subset V$ satisfies $\R\cl{C}\cap v_0^+K_1=\emptyset$, then $\# (v_0\Gamma\cap \R C )<\infty$. 
\end{theorem} 
\begin{proof}
The result can be deduced by the arguments as in the proof of Theorem~\ref{mctwo}; one may also use the basic ideas from \cite{OhShahGFH} about using the highest weight. 
\end{proof}

\end{document}